\documentclass[12pt, oneside, leqno]{article}

\usepackage{amsmath,amsthm}
\usepackage{amssymb}
\usepackage{yfonts}
\usepackage[utf8]{inputenc}
\usepackage{url}

\usepackage{enumerate}

\usepackage[T1]{fontenc}

\newtheorem{thm}{Theorem}
\newtheorem{cor}{Corollary}
\newtheorem{lem}{Lemma}

\newtheorem{prob}{Problem}

\theoremstyle{definition}

\newtheorem*{xrem}{Remark}
\newtheorem*{mainthm}{Main Theorem}
\newtheorem*{hist}{Historical Note}

\numberwithin{equation}{section}

\begin{document}

\title{Non-meagre subgroups of reals disjoint with meagre sets}

\author{
Ziemowit Kostana\\
Faculty of Mathematics, Informatics and Mechanics\\ 
Warsaw University\\
02-097 Warszawa, Poland\\
E-mail: z.kostana@mimuw.edu.pl}

\date{10.2017}

\maketitle

%% Classification and key words; note that the 2010 classification is used:

\renewcommand{\thefootnote}{}



\renewcommand{\thefootnote}{\arabic{footnote}}
\setcounter{footnote}{0}

%%%%%%%

\begin{abstract}
Let $(X,+)$ denote $(\mathbb{R}, +)$ or $(2^\omega, +_2)$. We prove that for any meagre set $F \subseteq X$ there exists a subgroup $G \le X$ without the Baire property, disjoint with some
translation of $F$. We point out several consequences of this fact and indicate why analogous result for the measure cannot be established in ZFC. We extend proof techniques from \cite{r-s}.
\end{abstract}

\section{Historical Background}

For sets $A,B \subseteq \mathbb{R}$ we define the algebraic sum $A+B=\{a+b|\, a\in A, \, b \in B\}$. Study of algebraic sums of this kind has been
around for almost a century. The first result in this topic seems to be due to Sierpi\'nski, who proved in 1920 that there exists two sets of measure zero whose sum is non-measurable \cite{Sie}. Rubel \cite{Rub} showed that these two sets can be chosen to be equal. This result was later generalized in many directions. For example, related results for other $\sigma$-ideals were obtained by Kharazishvili \cite{Kha} and by Cicho\'n and Jasi\'nski \cite{C-J}. In another direction, Ciesielski, Fejzi\'c and Freiling \cite{C-F} proved among others, that for every set $C\subseteq \mathbb{R}$, there exists a set $A \subset C$ such that $\lambda_*(A+A)=0$ and $\lambda^*(A+A)=\lambda^*(C+C)$, where $\lambda_*$ and $\lambda^*$ denote the inner and the outer Lebesgue measure respectively (but for simpler proof see the work by Marcin Kysiak \cite{kys}). It is also worth to mention the famous Erd{\"o}s-Kunen-Mauldin theorem \cite{EKM}.
\par It is easy to see that the sum of compact (open) sets is compact (open), the sum of $F_\sigma$ sets is $F_\sigma$, but for higher Borel classes this is not the case \cite{C-J}. Even the sum of a compact set with $G_\delta$ doesn't have to be a Borel set, and this was shown by Sodnomow in 1954 \cite{Sod} and independently by Erd\"os and Stone in 1970 \cite{S-E}.
\par The study of algebraic sums of subsets of real line is closely related to the study of additive subgroups of $(\mathbb{R},+)$. Erd\"os proved, that under CH there exists a non-meagre, null additive subgroup of reals, as well as a non-measurable, meagre additive subgroup of reals \cite{Erd}. The same can be proved under MA, but somehow surprisingly, while non-meagre subgroups of measure zero always exists, some additional set-theoretic assumption turns out to be necessary to prove the existence of a subgroup which is non-measurable and meagre. This was proved recently by Ros{\l}anowski and Shelah \cite{r-s}.

\par Results of this paper were obtained during the work on the author's Master's Thesis. Author would like to thank his advisor, dr. Marcin Kysiak, as well as prof. Witold Marciszewski and prof. Piotr Zakrzewski for many helpful remarks and suggestions.

\section{Preliminaries}

In private communication Sergei Akbarov posed the following problem, sometimes referred to as the Akbarov Problem.
\begin{prob}
Let $A\subseteq \mathbb{R}$ be a nonempty null set. Does there exist a set $B\subseteq \mathbb{R}$ with the property that $A+B$ is Lebesgue non-measurable?
\end{prob}
 One of the natural ways to approaching such problem is to try to find a non-measurable dense subgroup $G \le \mathbb{R}$ disjoint with some translation of $A$.
 Indeed, assume we have a non-measurable dense subgroup $G \le \mathbb{R}$, and $(A+v) \cap G =\emptyset$ holds. Then also 
 $$(A+v)\cap (G-G)=\emptyset,$$
 $$(A+v+G)\cap G=\emptyset.$$
 It is well-known (see for example \cite{buk}, Thm. 7.36) that every dense subgroup of $\mathbb{R}$ is either null or has full outer measure. Both $G$ and $A+v+G$ have full outer measure, hence
 both have inner measure zero, and so are non-measurable. 
 \par This approach, however sufficient to solve the problem with certain additional assumptions like Martin's Axiom, won't work in ZFC alone. In 2016 Andrzej Ros{\l}anowski and Saharon Shelah proved
the following theorem \cite{r-s}.
\begin{thm}
 It is relatively consistent with ZFC that any meagre subgroup of reals is null.
 \end{thm}
Indeed, consider a dense $G_\delta$ null subset of reals. Any subgroup disjoint with its translation must be meagre. So, consistently, also null.
\par In the case of Baire category however, situation is different. The following is the main result of this paper.
\begin{mainthm}
Let $X=(\mathbb{R},+)$ or $X=(2^\omega, +_2)$. For any meagre set $F \subseteq X$, there exists $x\in X$, and a dense subgroup $H \le X$ without the~Baire property such that $(F+x)\cap H=\emptyset$.
\end{mainthm}
From this follows the affirmative answer to the category version of Problem 1.
\begin{cor}
For any meagre set $A\subseteq X$, there exist a set $B \subseteq X$ such that $A+B$ doesn't have the Baire property.
\end{cor}

\begin{proof}
Just take as $B$ a dense subgroup without the Baire property, which is disjoint with a translation of $A$.
\end{proof}
Another consequence is
\begin{cor}
	There exists a null subgroup of $X$ which is not meagre.
\end{cor}

This was firstly proved by Talagrand \cite{tal}, and more recently by Ros{\l}anowski and Shelah \cite{r-s}.

\begin{proof}
	Just take $F$ in Theorem 2 of full measure.
\end{proof}

 In the whole paper $X$ denote the group of reals $(\mathbb{R},+)$ or the Cantor space $(2^\omega, +_2)$ with coordinate-wise addition modulo 2. Instead of $A+\{x\}$, we write $A+x$. $\mathcal{M}$ and $\mathcal{N}$ denote $\sigma$-ideals of meagre and null sets respectively.
A partition of $\omega$ is always a partition on finite intervals.
 \par The following quantifiers are commonly used in the infinite combinatorics: 
 \begin{itemize}
\item $\forall^\infty_{n<\omega}\; \psi(n)$, denoting ``$\psi$ holds for sufficiently large $n$'';
 \item $\exists^\infty_{n<\omega}\; \psi(n)$, denoting ``$\psi$ holds for infinitely many $n$''.
\end{itemize}
Here, one more similar notation will prove useful. Let $\textswab{U}$ be a fixed non-principial ultrafilter on $\omega$. Expression
$$\textswab{U}_{n<\omega} \; \psi(n)$$ will mean
$$\{n| \; \psi(n)\} \in \textswab{U}.$$
$\textswab{U}$ can be seen as something between $\forall^\infty$, and $\exists^\infty$. If $\forall^\infty_{n<\omega}\; \psi(n)$ and $\forall^\infty_{n<\omega}\; \phi(n)$, then clearly
$\forall^\infty_{n<\omega}\; \psi(n) \wedge\phi(n)$, and this is not the case for $\exists^\infty$. On the other hand, for any $\psi$ either $\exists^\infty_{n<\omega} \psi(n)$ or $\exists^\infty_{n<\omega} \neg \psi(n)$, and this is not the case for $\forall^\infty$.
It is straightforward from the definition of a non-principial ultrafilter that for $\textswab{U}$ both mentioned conditions holds.

 \par The following combinatorial characterization of meagre sets in $2^\omega$, due to Bartoszy\'nski (for the proof see \cite{b-j}, Thm. 2.2.4), will be crucial in our considerations.  
\begin{thm} Every meagre subset of $2^\omega$ is contained in a meagre set \\of the form
 
$$F=\{x \in 2^\omega |\, \forall^\infty_{n<\omega}\, x \restriction{I_n} \neq v \restriction{I_n}\},$$
where $\{I_n\}_{n<\omega}$ is a partition of $\omega$, and $v\in 2^\omega$.
 \end{thm}

\begin{xrem} It is not hard to see, that once we have the partition $\{I_n\}_{n<\omega}$, we can replace it with a ``thicker'' partition, i.e. one in which the end of every interval lies in some fixed infinite subset of $\omega$.
\end{xrem}

\begin{hist} H. Friedman and S. Shelah independently proved, that if a model $V$ results from adding $\omega_2$ Cohen reals to a model of $CH$, then in $V$ the following holds: if $E$ is an $F_{\sigma}$ subset of $X \times X$, which contains a rectangle of positive outer measure, then it contains a rectangle of positive measure. From this, M. Burke \cite{burke} concludes Theorem 1 as a corollary. For another reference, see \cite{pawl}.
\end{hist}

\section{Main theorem}

We turn to the proof of our main theorem. Firstly, we prove it for $X=2^\omega$, and then for $X=\mathbb{R}$, which will turn out to be more complicated. The following lemma was implicitly used in \cite{r-s} to obtain null, non-meagre subgroup of the Cantor Space, but in fact, there's more we can get from it.

\begin{lem} Let $\{I_n\}_{n<\omega}$ be a partition. Then $G=\{x \in 2^\omega | \, \textswab{U}_{n<\omega} \, x \restriction{I_n} \equiv 0 \}$ is a non-meagre dense subgroup of $2^\omega$.\end{lem}

\begin{proof} The fact that $G$ is a group is straightforward from properties of the ultrafilter. It is dense, since every sequence eventually equal 0 is in $G$. The non-trivial part is to show that it doesn't have the Baire property. It is well-known (see for example \cite{buk}, Thm. 7.38) that dense, proper subgroups of $2^\omega$ which have the Baire property are meagre, so it's enough to show that group $G$ is not meagre. Consider any $M \in \mathcal{M}(2^\omega)$. By virtue of Theorem 2 we can assume, possibly enlarging $M$, that

$$M=\{x \in 2^\omega |\, \forall_{k<\omega}^{\infty}\, x \restriction{J_k} \neq v \restriction{J_k} \},$$
 for some $v\in 2^\omega$ and a partition $\{J_n\}_{n<\omega}$. Moreover, applying Remark after Theorem 2, we can choose intervals $J_k$ in such a way that each of them is a finite sum of consecutive intervals of the form $I_r$, like below.

$$\overbrace{
\underbrace{0010}_{\text{$I_0$}}\underbrace{010}_{\text{$I_1$}}\underbrace{000111111}_{\text{$I_2$}}
}^{\text{$J_0$}}
\overbrace{
\underbrace{01101010100010}_{\text{$I_3$}}\underbrace{0010000101101}_{\text{$I_4$}}
}^{\text{$J_1$}}
\overbrace{
\underbrace{0001100001000}_{\text{$I_5$}}
\underbrace{11010011}_{\text{$I_6$}}
}^{\text{$J_2$}} \ldots 
$$

 \par Let now $$A_0=\{n <\omega | \, \exists_{r < \omega}\, I_n \subseteq J_{2r}\},$$ and 
 $$A_1=\{n <\omega | \, \exists_{r < \omega}\,  I_n \subseteq J_{2r+1}\}.$$ One of these sets belongs to the ultrafilter $\textswab{U}$. Suppose it's $A_0$.
Then we put  
\begin{equation*}
    x \restriction{I_k}=
    \begin{cases}
      0, & \text{if}\ k \in A_0 \\
      v \restriction{I_k}, & \text{if}\ k \in A_1.
    \end{cases}
  \end{equation*} Similarily, if $A_1 \in \textswab{U}$, we put
\begin{equation*}
    x \restriction{I_k}=
    \begin{cases}
      0, & \text{if}\ k \in A_1 \\
      v \restriction{I_k}, & \text{if}\ k \in A_0.
    \end{cases}
  \end{equation*} In any case, $x$ is constructed in such a way that $\textswab{U}_{n<\omega} \, x \restriction{I_n} \equiv 0$, but also $\exists^\infty_{n<\omega}\; x \restriction{J_n}=v \restriction{J_n}$. In fact, $\{n < \omega | \, x \restriction{J_n}=v\restriction{J_n}\}$ is either the set of even or the set of odd non-negative integers. This means that $x \in G \setminus M$, and since $M$ was arbitrary, this shows that $G \notin \mathcal{M}.$
 \end{proof}

\begin{thm}
For any set $F \in \mathcal{M}$, there exist a $x\in X$ and a dense subgroup $H \le X$ without the Baire property such that $(F+x)\cap H=\emptyset$.
\end{thm}

\subsection{Version for $2^\omega$} 

Take any $F \in \mathcal{M}(2^\omega)$. Using Theorem 2, we can assume that $$F = \{x \in 2^\omega |\, \forall_{k< \omega}^{\infty}\, x \restriction{I_k} \neq x_F \restriction{I_k} \},$$ where $x_F \in 2^\omega$, and $\{I_n\}_{n<\omega}$ is a partition of $\omega$. 
Then, we only have to notice that $\{x \in 2^\omega | \, \textswab{U}_{n<\omega} \, x \restriction{I_n} \equiv 0 \} \cap (F+_2x_F)=\emptyset$ and use Lemma 1. To this end, see that
$$\{x \in 2^\omega | \, \textswab{U}_{n<\omega} \, x \restriction{I_n} \equiv 0 \} \subseteq \{x \in 2^\omega | \, \exists^\infty_{n<\omega} \, x \restriction{I_n} \equiv 0\},$$
and
$$F+_2x_F \subseteq \{x \in 2^\omega | \, \forall^\infty_{n<\omega} \, x \restriction{I_n} \neq 0 \}.$$

\subsection{Version for $\mathbb{R}$}

 For any (finite or infinite) binary sequence $x$, we denote by $x^{op}$ the sequence obtained from $x$ by changing every $0$ to $1$, and vice versa. Let $D^{\infty}=\{x \in~ 2^\omega |\, \exists^{\infty}_{n<\omega} \, x(n)=0\}$. It is known that there exists the continuous bijection preserving measure and category, $\phi : D^{\infty} \rightarrow [0,1)$, given by the formula
$$\phi(x)=\sum_{i=0}^\infty{\frac{x(i)}{2^{i+1}}}.$$
 
\par Consider any $F \in \mathcal{M(\mathbb{R})}$. Replacing if needed $F$ by $F+\mathbb{Z}$, we can assume that $F+\mathbb{Z}=F$. Let $\widetilde{F}=\phi^{-1}[F \cap [0,1)]$. There exists a $w \in 2^{\omega}$ and a partition of $\omega, \; \{I_n\}_{n<\omega}$, with the property that
\begin{equation} \widetilde{F} \subseteq \{x \in 2^{\omega}|\, \forall^{\infty}_{n<\omega}\, x \restriction{I_n} \neq w \restriction{I_n} \}. \end{equation}

 In this case, we'll need a combinatorial characterization with some stronger properties. We choose another partition of $\omega$, $\{J_n\}_{n<\omega}$, satisfying the following conditions:
\begin{itemize}
\item  each interval $J_n$ is a sum of at least three consecutive intervals $I_k$;
\item  if $J_n=I_{r_0} \cup \ldots \cup I_{r_{k}}$, and indices are increasing, then \\
$|I_{r_2} \cup \ldots \cup I_{r_k}| > |I_{r_0} \cup I_{r_1}|;$
\item  sequence $\{|J_n|\}_{n<\omega}$ is increasing.
\end{itemize}

We now construct the sequence $v \in 2^{\omega}$. For every $n < \omega$, we write $$J_n=I_{r_0} \cup \ldots \cup I_{r_{k}}$$ as a sum of consecutive intervals of the form $I_k$, with increasing indices, and put:
 $$v \restriction{I_{r_0}} = w\restriction{I_{r_0}},$$
 $$v\restriction{I_{r_1}}=w^{op}\restriction{I_{r_1}},$$ 
 $$v\restriction{I_{r_2} \cup \ldots \cup I_{r_k}}=010101...01(0).$$

 Given (3.1), it is evident that 
\begin{equation}\widetilde{F}\subseteq \{x \in 2^{\omega}| \, \forall^{\infty}_{n<\omega}\, x \restriction{J_n} \neq v \restriction{J_n} \}. \end{equation}
Let $J_n=[a_n,b_n]$ for every $n$. 

\begin{lem}Assume we have sequences $x,y,z \in D^\infty$ such that $x\restriction{[a,b]} \equiv 0$, $y(b)=0$,  and $\phi(z)+\rho=\phi(x) + \phi(y)$, where $\rho \in \{0,1\}$. Then $y\restriction{[a,b-1]}=z\restriction{[a,b-1]}$.
\end{lem}
\begin{proof}
\begin{align*}
\phi(x)+\phi(y) = \sum_{i=0}^{a-1}{\frac{x(i)}{2^{i+1}}}+
\sum_{i=a}^{b}{\frac{x(i)}{2^{i+1}}}+\sum_{i=b+1}^{\infty}{\frac{x(i)}{2^{i+1}}}&+\\
\sum_{i=0}^{a-1}{\frac{y(i)}{2^{i+1}}}+\sum_{i=a}^{b}{\frac{y(i)}{2^{i+1}}}+\sum_{i=b+1}^{\infty}{\frac{y(i)}{2^{i+1}}} &=\\
 \sum_{i=0}^{a-1}{\frac{x(i)+y(i)}{2^{i+1}}}+\sum_{i=a}^{b-1}{\frac{y(i)}{2^{i+1}}}+ 
\sum_{i=b+1}^{\infty}{\frac{x(i)+y(i)}{2^{i+1}}} &=\\ \sum_{i=0}^{\infty}{\frac{z(i)}{2^{i+1}}}+\rho.
\end{align*}
$\displaystyle{\sum_{i=b+1}^{\infty}{\frac{x(i)+y(i)}{2^{i+1}}} < \sum_{i=b+1}^{\infty}{\frac{1}{2^i}}=2^{-b}}$, so this part of the sum does not affect positions of $z$ with indices less than $b$.
\end{proof}

Let us now define
$$\widetilde{G}=\{x \in 2^{\omega} |\, \exists_{m<\omega} \, \textswab{U}_{n<\omega} \, x\restriction{[a_n,b_n-m]}\equiv 0\}.$$ 

In fact, we didn't rule out the possibility that $m>b_n-a_n$, but recall that $|a_n-~b_n|\rightarrow~ \infty$, so this situation can only occur on finitely many positions. This is clearly a~ subgroup of $2^{\omega}$, and $\widetilde{G}
 \supseteq \{x \in 2^\omega | \, \textswab{U}_{n<\omega} \, x \restriction{J_n} \equiv 0 \}$,
which, by Lemma~ 1, is not meagre, so we obtain that $\widetilde{G}$ is not meagre.
\par Group $\widetilde{G}$ defined this way can be ``transferred'' to $\mathbb{R}$, and in fact authors of \cite{r-s} apply this kind of idea. 
\begin{lem} $\phi[\widetilde{G}] +\mathbb{Z}$ is a subsemigroup of $\mathbb{R}$. 
\end{lem}
\begin{proof}
Assume we have an equality $$\phi(g)+k_1+\phi(h)+k_2=\phi(f)+l,$$ where $g,h \in \widetilde{G}, \, k_1,k_2,l \in \mathbb{Z}$. We shall show that $f \in \widetilde{G}$. 
\par There exist integers $m_1,m_2$ such that $$\textswab{U}_n\, g\restriction{[a_n,b_n-m_1]\equiv 0}, \; \textswab{U}_n \, h\restriction{[a_n,b_n-m_2]\equiv 0},$$
so that if we set $\overline{m}$=max$\{m_1,m_2\}$, we obtain
$$\textswab{U}_n\, g\restriction{[a_n,b_n-\overline{m}]} \equiv h\restriction{[a_n,b_n-\overline{m}]} \equiv 0.$$
Then, from Lemma 2 we conclude that $\textswab{U}_n \, f \restriction[a_n,b_n-(\overline{m}+1)]\equiv 0$, which shows that $f \in \widetilde{G}$.
\end{proof}

\begin{lem} $\widetilde{G}+_2v, \, \widetilde{G}+_2v^{op}\subseteq D^\infty$, \\
 moreover $\phi[\widetilde{G}+_2v]+\mathbb{Z} \supseteq \phi[\widetilde{G}] + \phi[\widetilde{G}+_2v] + \mathbb{Z}$, \\and $\phi[\widetilde{G}+_2v^{op}]+\mathbb{Z} \supseteq \phi[\widetilde{G}] + \phi[\widetilde{G}+_2v^{op}] + \mathbb{Z}$.
\end{lem}
\begin{proof}
The first part easily follows from the definitions. For the proof of the~~second part, assume we have a $z \in D^\infty$ with the property that $\phi(z)+\rho=\phi(g_1)+\phi(g_2+_2v)$, for some $g_1, g_2 \in \widetilde{G}$ and $\rho \in \{0,1\}.$ We'll show that $z \in \widetilde{G}+_2v$. 
\par Let $m_1,m_2 < \omega$ satisfy 
$$\textswab{U}_{n<\omega} \, g_i\restriction{[a_n,b_n-m_i]} \equiv 0,$$
for $i=0,1$, and $\overline{m}=\text{max}\{m_1,m_2\}$. Then 
$$\textswab{U}_{n<\omega}\, (g_2+_2v)\restriction{[a_n,b_n-\overline{m}]}=v\restriction{[a_n,b_n-\overline{m}]}$$ and
$$\textswab{U}_{n<\omega}\, g_1\restriction{[a_n,b_n-\overline{m}]}\equiv 0.$$

 Because $|a_n-b_n| \rightarrow \infty$, for sufficiently large $n$, $\frac{a_n+b_n}{2}\le b_n-\overline{m},$ which means that $b_n-\overline{m}$ lies in the right half of the interval $J_n$. Recall that $v$ was defined in such a way that in the right half of every interval $J_n$,  it is of the~form $0101010 \ldots1(0)$. Thus for sufficiently large $n\text{, either} \; v(b_n-\overline{m})=0$ or $v(b_n-(\overline{m}+1))=0$. From Lemma 2 (putting $z=z, x=g_1, y= v$), we infer that 
$$\textswab{U}_{n<\omega}\, z\restriction{[a_n,b_n - (\overline{m}+2)]}= v\restriction{[a_n,b_n -(\overline{m}+2)]}.$$
In conclusion $z+_2v\in \widetilde{G}$, which is exactly $z \in \widetilde{G}+_2v$. For  $v^{op}$ proof is the same, except that we replace every instance of $v$ with $v^{op}$. 
\end{proof}

Notice now, that
\begin{equation}(\widetilde{G}+v) \cap \widetilde{F} = \emptyset. \end{equation} 
This is true, because of (3.1) and the inclusion
$$(\widetilde{G}+v) \subseteq \{x \in 2^\omega | \, \exists^\infty_{n<\omega} \, x \restriction{I_n}=w\restriction{I_n}\}.$$
Set $H=\phi[\widetilde{G}]-\phi[\widetilde{G}]+\mathbb{Z}$. From Lemma 3 directly follows that it is a subgroup of $\mathbb{R}$. What's left, is to verify that
$$(\phi[\widetilde{G}]-\phi[\widetilde{G}]+\mathbb{Z}) \cap (F-\phi(v))=\emptyset.$$ 

Suppose to the contrary, that for some $g,h \in \widetilde{G}, \; k \in \mathbb{Z} \text{ and } f\in F$, an~equality

$$\phi(g)-\phi(h)+\phi(v)+k=f$$ 
holds.
From Lemma 4 we know, that $\phi(g)+\phi(v)+k=\phi(g'+_2v)+k'; \\ g'\in \widetilde{G},\; k' \in \mathbb{Z},$ so we obtain
\begin{equation}\phi(g'+_2v)-\phi(h)+k'=f.\end{equation}
One can verify that for any sequence $x \in D^\infty$, for which also $x^{op} \in D^\infty$, identity $\phi(x)+\phi(x^{op})=1$ holds. This allows us to write
\begin{equation}\phi(g'+_2v)=1-\phi((g'+_2v)^{op})=1-\phi(g'+_2v^{op}),\end{equation}

because $^{op}$ operation is just the addition of a constant sequence, and so plugging (3.5) into (3.4) yields

\begin{align*}f=1-\phi(g'+_2v^{op})-\phi(h)+k' &=\\
1-[\phi(g'+_2v^{op})+\phi(h)]+k' &=\\
k''+1-\phi(\overline{g}+_2v^{op}) &=\\
k''+\phi(\overline{g}+_2v).
\end{align*}
We used Lemma 4 in the third equality, and $\phi(\overline{g}+_2v)+\phi(\overline{g}+_2v^{op})=1$ in the last one.
Since $F+\mathbb{Z}=F$, we conclude that
$$\phi(\overline{g}+_2v) \in F\cap[0,1)=\phi[\widetilde{F}],$$
$$\overline{g}+_2v \in \widetilde{F},$$
which contradicts (3.3). Finally, let us notice that during the whole calculation arguments of $\phi$ were always within its domain, $D^\infty$.

\section{Open problems}

The proof of Theorem 2 uses the characterization of sets from the $\mathcal{M}(2^\omega)$ ideal, which turns out to match very well with the properties of ultrafilters. There's another class of sets with elegant combinatorial characterization, namely $\mathcal{E}$ -- the $\sigma$-ideal generated by closed sets of measure zero. The following theorem is due to Bartoszy\'nski and Shelah and is proved in \cite{b-j}.
\begin{thm}
$E \in \mathcal{E}(2^\omega)$ if and only if
$$E \subseteq \{x \in 2^\omega | \; \forall^\infty_{n<\omega} \;  x \restriction{I_n} \in K_n \},$$
where $\{I_n\}_{n<\omega}$ is a partition of $\omega$, $K_n \subseteq 2^{I_n}$ and $\forall_{n<\omega}\, \frac{|K_n|}{2^{|I_n|}} \le 2^{-n}$.
\end{thm}
  The following seems to be a reasonable question.
\begin{prob}
Let $E \in \mathcal{E}(2^\omega)$. Does there necessarily exists a dense non-measurable subgroup $G\le 2^\omega$, disjoint with some translation of $E$?
\end{prob}

Another related question was asked by Taras Banakh on the Mathoverflow. 

\begin{prob}
	Does there exist (in ZFC) a subgroup $G \le X$, such that \\$G \in \mathcal{N}\cap\mathcal{M}$, but $G \notin \mathcal{E}$?
\end{prob}

\end{document}